\documentclass{amsart}
\usepackage{amsfonts,amssymb,amsmath,amsthm}
\usepackage{url}
\usepackage{enumerate}
\usepackage{amssymb}

\newtheorem{theorem}{Theorem}[section]
\newtheorem{lem}[theorem]{Lemma}

\newtheorem{con}[theorem]{Conjecture}
\theoremstyle{definition}

\newtheorem{remark}[theorem]{Remark}
\numberwithin{equation}{section}

\DeclareMathOperator{\dia}{diag}

\newcommand{\eps}{\varepsilon}
\newcommand{\N}{\ensuremath{{\mathbb N}}}

\newcommand{\R}{\ensuremath{{\mathbb R}}}
\newcommand{\E}{\ensuremath{{\mathbb E}}}
\newcommand{\Pro}{\ensuremath{{\mathbb P}}}
\newcommand{\norm}[1]{\left \lVert#1 \right\rVert}
\newcommand{\abs}[1]{\left\lvert#1 \right\rvert}

\newcommand{\vertiii}[1]{{\left\vert\kern-0.25ex\left\vert\kern-0.25ex\left\vert #1
		\right\vert\kern-0.25ex\right\vert\kern-0.25ex\right\vert}}

\author{Olivier Gu\'edon\and Aicke Hinrichs  \and Alexander E. Litvak \and Joscha Prochno}
\address[Olivier Gu\'edon]{Laboratoire d'Analyse et de Math\'ematiques Appliqu\'ees\\
	Universit\'e Paris-Est\\
    77454 Marne-la-Vall\'ee, Cedex 2\\
	France}
\email{olivier.guedon@u-pem.fr}
\address[Alexander E. Litvak]{Department of Mathematical and Statistical Sciences\\
University of Alberta\\
Edmonton, AB, T6G 2G1\\
Canada}
\email{alitvak@ualberta.ca}
\address[Aicke Hinrichs]{Institut f\"ur Analysis\\
Johannes Kepler Universit\"at Linz\\
Altenbergerstrasse 69\\
4040 Linz\\
Austria}
\email{aicke.hinrichs@jku.at}
\address[Joscha Prochno]{Department of Mathematics\\
	University of Hull\\
	Cottingham Road\\
	Hull, HU6 7RX\\
	United Kingdom}
\email{j.prochno@hull.ac.uk}
\thanks{J.P. was supported in part by the Austrian Science Fund grant FWFM 1628000 ``Local Theory of Banach Spaces and Convex Geometry''.}


\begin{document}

\title[Operator norms of random matrices]{On the expectation of operator norms of random matrices}

\begin{abstract}
We prove estimates for the expected value of operator norms of Gaussian random matrices with independent and mean-zero entries, acting as operators from $\ell^n_{p^*}$ to $\ell_q^n$, $1\leq p^* \leq 2 \leq q \leq \infty$.
\end{abstract}
\maketitle


\section{Introduction and main results}

Random matrices and their spectra are under intensive study in Statistics since the work of Wishart \cite{Wish} on sample covariance matrices, in Numerical Analysis since their introduction by von Neumann and Goldstine \cite{vN} in the $1940$s, and in Physics as a consequence of Wigner's work \cite{Wi1,Wi2} since the $1950$s. His Semicircle Law, a fundamental theorem in the spectral theory of large random matrices describing the limit of the empirical spectral measure for what is nowadays known as Wigner matrices, is among the most celebrated results of the theory.

In Banach Space Theory and Asymptotic Geometric Analysis, random matrices appeared already in the $70$s (see e.g.  \cite{BGN,BGor, G-PTRF}). In \cite{BGN}, the authors obtained
asymptotic bounds for the expected value of the operator norm of a random matrix $B=(b_{ij})_{i,j=1}^{m,n}$ with independent mean-zero entries with $|b_{ij}|\leq 1$ from $\ell_2^n$ to $\ell_q^m$, $2 \leq q<\infty$. To be more precise, they proved that
\[
\E\, \big\|B:\, \ell_2^n \to \ell_q^m\big\| \leq C_q \cdot \max\{m^{1/q},\sqrt{n} \},
\]
where $C_q$ depends only on $q$.
This was then successfully used to characterize $(p,q)$-absolutely summing operators on Hilbert spaces. Ever since, random matrices are extensively studied and methods of Banach spaces  have produced numerous deep and new results.
In particular, in many applications the spectral properties of a Gaussian
matrix, whose entries are independent identically distributed (i.i.d.) standard Gaussian random variables, were  used. Moreover,  Seginer  \cite{Seg} proved  that for an $m\times n$ random matrix with i.i.d. symmetric random variables the expectation of its spectral norm (that is, the operator norm  from $\ell_2^n$ to $\ell_2^m$) is of the order of the expectation of the largest Euclidean norm of its rows and columns. He also proved an optimal result in  the case of random matrices with entries $\eps_{ij} a_{ij}$ where $\eps_{ij}$ are independent Rademacher random variables and $a_{ij}$ are fixed numbers.
We refer the interested reader to the survey \cite{DS-sur1, DS-sur2} and references therein.

It is natural to ask similar questions about general
random matrices, in particular about Gaussian matrices whose entries are still independent centered Gaussian random variables, but with different variances. In this case, where we drop the assumption of identical distributions, very little is known.
It is conjectured that the expected spectral norm of such a Gaussian matrix
is as in Seginer's result, that is, of the order of the expectation of the largest Euclidean norm of its rows and columns. A big step toward the solution was made by Lata{\l}a in \cite{L}, who proved a bound involving fourth moments, which is of the right order $\max\{ \sqrt{m},\sqrt{n}\}$ in the i.i.d. setting. On one hand, in view of the classical Bai-Yin theorem, the presence of fourth moments is not surprising, on the other hand they are not needed if the conjecture is true.

Later in \cite{RS}, Riemer and Sch\"utt proved the conjecture up to a $\log n$ factor. The two results are incomparable -- depending on the choice of variances, one or another gives a better bound. The Riemer-Sch\"utt estimate was used recently in \cite{RZ}.

Another big step toward the solution was made a short while ago by Bandeira and Van Handel \cite{BvanH}. In particular, they proved that
\begin{equation}
\label{eq:BvanH}
\E\, \big\| (a_{ij} g_{ij}) : \ell_2^m \to \ell_2^n \big\| \le C \Big( \vertiii A  + \sqrt{\log \max(n,m)}  \max_{ij} | a_{ij}| \Big)
\end{equation}
where $\vertiii A$ denotes the largest Euclidean norm of the rows and columns of $(a_{ij})$, $C>0$ is a universal constant,  and $g_{ij}$ are independent standard Gaussian random variables. This was further developed by Van Handel \cite{vanH} to verify
the conjecture up to a $\log \log n$  factor. In fact, more was proved in \cite{vanH}. He computed precisely
the expectation of the largest Euclidean norm of the rows and columns using Gaussian concentration. And, while the moment method is at the heart of the proofs in  \cite{Seg} and \cite{BvanH}, he proposed a very nice approach based on the comparison of Gaussian processes to improve the result of Lata\l a. However, his approach is based on a trick using symmetric matrices and this cannot be generalized to study other operator norms of the matrix.

The purpose of this work is to provide some bounds for operator norms of such matrices considered as operators from $\ell_p^m$ to $\ell_q^n$.

In what follows, by $g_i$, $g_{ij}$, $i\geq 1$, $j\geq 1$ we always denote  independent standard Gaussian random variables.
Let $n,m\in\N$ and $A=(a_{ij})_{i,j=1}^{m,n} \in\R^{m\times n}$. We write $G=G_A=(a_{ij}g_{ij})_{i,j=1}^{m,n}$. For $r\geq 1$, we denote by $\gamma_r\approx \sqrt{r}$ the $L_r$-norm  of a standard Gaussian random variable. The notation $f\approx h$ means that there are two absolute
positive constants $c$ and $C$ (that is, independent of any parameters)
such that $c f\leq h \leq C f$.


%
%

Our main result is the following theorem.

\begin{theorem}\label{thm:main}
	For every $1<p^*\leq 2 \leq q < \infty$ one has
	\begin{align*}
	\E\, \big\|G:\ell_{p^*}^m\to\ell_q^n\big\|
    &\leq \Big(\E\, \big\|G:\ell_{p^*}^m\to\ell_q^n\big\|^q \Big)^{1/q} \\
	& \leq C\, p^{5/q}\, (\log m)^{1/q}\, \bigg[\, \gamma_p\,
     \max_{i \leq m}\|(a_{ij})_{j=1}^n\|_p + \gamma_q\,
     \E\max_{i\leq m\atop j\leq n}|a_{ij}g_{ij}| \,\bigg] \\
	& + 2^{1/q}\, \gamma_q \, \max_{j \leq n}\|(a_{ij})_{i=1}^m\|_q,
	\end{align*}
where $C$ is a positive absolute constant.
\end{theorem}

We conjecture the following bound.

\begin{con}
	For every $1\leq p^*\leq 2 \leq q \leq \infty$ one has
	\begin{align*}
	\E\, \big\|G:\ell_{p^*}^m\to\ell_q^n\big\|
	& \approx_{p,q} \max_{ i \leq m}\|(a_{ij})_{j=1}^n\|_p
	+  \max_{ j \leq n}\|(a_{ij})_{i=1}^m\|_q +
    \E\max_{i\leq m\atop j\leq n}|a_{ij}g_{ij}|.
	\end{align*}
\end{con}
The notation $f\approx _{p,q} h$ means that there are two
positive constants $c(p, q)$ and $C(p,q)$, which depend only on
the parameters $p$ and $q$, such that $c(p, q) f\leq h \leq C(p, q) f$, and,
as usual $1/p+1/p^*=1$.
This conjecture extends the corresponding conjecture for the case $p=q=2$
and $m=n$. In this case, Bandeira and Van Handel \cite{BvanH} (see  \eqref{eq:BvanH})
proved an estimate with $\sqrt{\log \max(m,n)} \max |a_{ij}|$ instead of  $\E\max |a_{ij}g_{ij}|$,
while in  \cite{vanH} the corresponding bound is proved with  $\log \log n$ in front of
the right hand side.

\begin{remark}
	Note that if $1\leq p^* \leq 2 \leq q \leq \infty$, in the case of matrices of tensor structure, that is, $(a_{ij})_{i,j=1}^n = x\otimes y = (x_j\cdot y_i)_{i,j=1}^n$, with $x,y\in\R^n$, Chevet's theorem \cite{Chev,BGor} and a direct computation show that
	\[
	\E\,\big\|G:\ell^n_{p^*}\to\ell^n_{q}\big\|
	\approx_{p,q} \|y\|_q\|x\|_\infty + \|y\|_\infty\|x\|_p.
	\]
	If the matrix is diagonal, that is, $(a_{ij})_{i,j=1}^n = \dia(a_{11},\dots,a_{nn})$, then we immediately obtain
		\[
		\E\,\big\|G:\ell_{p^*}^n\to \ell_q^n\big\| = \E\,\|(a_{ii}g_{ii})_{i=1}^n\|_{\infty}\approx \max_{i\leq n }
		\sqrt{\ln(i+3)} \,\cdot a_{ii}^* \approx \|(a_{ii})_{i=1}^n\|_{M_g},
		\]
		where $(a_{ii}^*)_{i\leq n}$ is the decreasing rearrangement
		of $(|a_{ii}|)_{i\leq n}$ and $M_g$ is the Orlicz function given by
		\[
		M_g(s) = \sqrt{\frac{2}{\pi}}\int_0^s e^{-\frac{1}{2t^2}} \,dt
		\]
		(see Lemma \ref{max-gaus} below and \cite[Lemma 5.2]{GLSW2} for the Orlicz norm expression).
		
		Slightly different estimates of the same flavour can be also obtained in the case $1\leq q \leq 2 \leq p^* \leq \infty$.
\end{remark}

\section{Notation and Preliminaries}


By $c, C, C_1, ...$ we always denote positive absolute constants,
whose values may change from line to line, and we write $c_p, C_p, ...$
if constants depend on some parameter $p$.



Given $p\in [1, \infty]$, $p^*$ denotes its conjugate, that is $1/p + 1/p^*=1$.
For $x=(x_i)_{i\leq n}\in \R^n$, $\|x\|_p$ denotes its $\ell_p$-norm, that is $\|x\|_{\infty}=\max_{i\leq n} |x_i|$ and, for $p<\infty$,
$$
  \|x\|_p = \Big(\sum_{i=1}^n |x_i|^p \Big)^{1/p}.
$$
The corresponding space $(\R^n,\|\cdot\|_p)$ is denoted by $\ell_p^n$, its unit ball by $B_p^n$.

If $E$ is a normed space, then $E^*$ denotes its dual space and $B_E$ its closed unit ball. The modulus of convexity of $E$ is defined for any $\varepsilon\in (0,2)$ by
\[
\delta_E(\varepsilon):=\inf\Big\{ 1-\Big\| \frac{x+y}{2} \Big\|_E \,:\, \|x\|_E=1,~ \|y\|_E=1,~ \|x-y\|_E>\varepsilon \Big\}.
\]
We say that $E$ has modulus of convexity of power type 2 if there exists a positive constant $c$ such that for all $\varepsilon\in(0,2)$, $\delta_E(\varepsilon)\geq c\varepsilon^2$. It is well known that this property (see e.g. \cite{F} or \cite[Proposition 2.4]{Pis75}) is equivalent to the fact that
\[
\Big\| \frac{x+y}{2} \Big\|_E^2 + \lambda^{-2}\Big\| \frac{x-y}{2} \Big\|_E^2 \leq \frac{\|x\|_E^2+\|y\|_E^2}{2}
\]
holds for all $x,y\in E$, where $\lambda>0$ is a constant depending only on $c$. In that case, we say that $E$ has modulus of convexity of power type 2 with constant $\lambda$.  We clearly have $\delta_E(\varepsilon) \geq \varepsilon^2/(2\lambda^2)$. It follows from Clarkson's inequality \cite{Clarkson} that for $p>2$ the space
$\ell_{p^*}^n$ has modulus of convexity of power type 2 with
\[
	\lambda^{-2}=\frac{p^*(p^*-1)}{8} \approx \frac{1}{p}.
\]
Recall that a Banach space $E$ is of Rademacher type $r$ for some $1\leq r \leq 2$ if there is $C>0$ such that for all $n\in \mathbb N$ and for all $x_1,\dots,x_n\in E$,
\[
\bigg(\mathbb E_{\varepsilon}\Big\|\sum_{i=1}^n\varepsilon_i x_i\Big\|^2\bigg)^{1/2} \leq C \left( \sum_{i=1}^n\|x_i\|^r \right)^{1/r},
\]
where $(\varepsilon_i)_{i=1}^\infty$ is a sequence of independent random variables defined on some probability space $(\Omega, \Pro)$
such that $\Pro(\varepsilon_i=1)=\Pro(\varepsilon_i=-1)=\frac{1}{2}$ for every $i\in\N$. The smallest $C$ is called type-$r$ constant of $E$, denoted by $T_r(E)$. This concept was introduced into Banach space theory by Hoffmann-J\o rgensen \cite{HJ} in the early $1970$s and the basic theory was developed by Maurey and Pisier \cite{MaureyPisier}.

We will need the following theorem, which follows from \cite{GR} in combination with an improvement that can be found in \cite{GMPT}.

\begin{theorem}\label{thm:GR}
	Let $E$ be a Banach space with modulus of convexity of power type 2 with constant $\lambda$. Let $X_1,\dots,X_m\in E^*$ be independent random vectors. For $q\geq 2$, if
	\[
	B:= C \lambda^4T_2(E^*)\sqrt{\frac{\log m}{m}}\Big( \E \max_{1\leq i\leq m}\|X_i\|_{E^*}^q\Big)^{1/2},
	\]
	where $T_2(E^*)$ is the type 2 constant of $E^*$, and
\[
	\sigma :=\sup_{y\in B_E}\left( \frac{1}{m} \sum_{i=1}^m
     \E |\langle X_i,y \rangle|^q\right)^{1/q} ,
\]
	then
\begin{align*}
	\E \sup_{y\in B_E} \bigg| \frac{1}{m}\sum_{i=1}^m |\langle X_i,y
    \rangle|^q - \E|\langle X_i,y \rangle|^q\bigg|
	& \leq A^2 + A\cdot \sigma^{q/2} .
	\end{align*}
 \end{theorem}

We also recall known facts about Gaussian random variables.
The next lemma is well-known (see e.g. Lemmas~2.3, 2.4 in \cite{vanH}).

\begin{lem}\label{max-gaus}
Let $a=(a_i)_{i\leq n} \in \R^n$ and $(a_i^*)_{i\leq n}$ be the decreasing rearrangement
of $(|a_i|)_{i\leq n}$.
Then
\[
  \E\, \max_{i\leq n} |a_i g_i| \approx \max_{i\leq n }
   \sqrt{\ln(i+3)} \,\cdot a_i^*.
\]
\end{lem}
Note that in general the maximum of i.i.d. random variables weighted by coordinates of a vector
$a$ is equivalent to a certain Orlicz norm
$\|a\|_M$, where the function $M$ depends only on the distribution
of random variables (see \cite[Corollary 2]{GLSW-ann} and Lemma~5.2 in \cite{GLSW2}).

The following theorem is the classical Gaussian concentration
inequality (see e.g. \cite{CIS} or inequality (2.35) and
Proposition~2.18 in \cite{Led}).

\begin{theorem}\label{thm:concentration}
Let $n\in\N$ and $(Y,\norm{\cdot}_Y)$ be a Banach space. Let $y_1,\ldots,y_n\in Y$ and  $X=\sum_{i=1}^n g_i y_i$. Then, for every $t>0$,
\begin{equation}\label{eq:gaussian concentration}
\Pro \Big( \big|\norm{X}_Y - \E\norm{X}_Y\big| \geq t\Big) \leq 2
\exp\left(
-{\frac{t^2}{2 \sigma_Y(X)^2}}\right),
\end{equation}
where $\sigma_Y (X) = \sup_{\|\xi\|_{Y^*}=1} \left( \sum_{i=1}^n \abs{\xi(y_i)}^2 \right)^{\frac{1}{2}}$.
\end{theorem}

\begin{remark}\label{re:estimate sigma}
Let $p>2$ and $p^* \leq q\leq \infty$. Let $a=(a_j)_{j\leq n}\in \R^n$ and $X=(a_j g_j)_{j\leq n}$. Then we clearly have
	\[
	\sigma _{\ell_p^n} (X)  =  \max_{j \leq n}|a_j|.
	\]
Thus, Theorem~\ref{thm:concentration} implies for $X=(a_j g_j)_{j\leq n}$
	\begin{align}\label{gaus-conc}
	\Pro \Big(\big| \|X\|_p - \E\|X\|_p \big| > t\Big) \leq 2\, \exp\bigg( - \frac{t^2}{2 \max_{j \leq n}|a_{j}|^2} \bigg).
	\end{align}
Note also that
\begin{align}\label{norm-p}
	\E \|X\|_p  \leq
   \bigg(   \E\, \sum_{j=1}^n |a_j|^p |g_j|^p\bigg)^{1/p}
  =    \gamma_p \|a\|_p.
	\end{align}
\end{remark}

\section{Proof of the main result}

We will apply Theorem \ref{thm:GR} with  $E=\ell_{p^*}^n$, $1<p^* \leq 2$
and $X_1,\dots,X_m$ being the rows of the matrix $G=(a_{ij}g_{ij})_{i,j=1}^{m,n}$.
We start with two lemmas in which we estimate the quantities $\sigma$ and the expectation, appearing in that theorem.

\begin{lem}\label{lem:estimate sigma}
	Let $m,n\in \N$, $1<p^* \leq 2 \leq q$, and for $i\leq m$ let $X_i=(a_{ij}g_{ij})_{j=1}^n$. Then
	\begin{align*}\label{eq:value sigma}
	\sigma& =\sup_{y\in B_{p^*}^n}\bigg(\frac{1}{m}\sum_{i=1}^m \E \big| \langle X_i,y  \rangle\big|^q \bigg)^{1/q} = \frac{\gamma_q}{m^{1/q}}\, \max_{j \leq n}\|(a_{ij})_{i=1}^m\|_q.
	\end{align*}
\end{lem}
\begin{proof}
For every $i\leq m$, $\langle X_i ,y \rangle = \sum_{j=1}^na_{ij}y_jg_{ij}$,
is a Gaussian random variable with variance $\|(a_{ij}y_j)_{j=1}^n\|_2$. Hence,
\begin{align*}
\sigma^q& =\sup_{y\in B_{p^*}^n}\frac{1}{m}\sum_{i=1}^m \E|\langle X_i,y\rangle|^q
= \frac{\gamma_q^q}{m}\sup_{y\in B_{p^*}^n}\sum_{i=1}^m\bigg(\sum_{j=1}^n|a_{ij}y_j|^2\bigg)^{q/2}.
\end{align*}

Taking standard unit vectors $e_1,\dots,e_n\in B_{p^*}^n$, we immediately obtain
\[
\sup_{y\in B_{p^*}^n}\sum_{i=1}^m\bigg(\sum_{j=1}^n|a_{ij}y_j|^2\bigg)^{q/2}
\geq \max_{j \leq n} \|(a_{ij})_{i=1}^m\|_q^q.
\]
which proves the lower bound for $\sigma$. The corresponding upper bound is a consequence of H\"older's inequality. Indeed, since $q\geq 2$,
\begin{align*}
\sum_{j=1}^na_{ij}^2y_j^2 & = \sum_{j=1}^na_{ij}^2y_j^{2p^*/q}y_j^{2-2p^*/q}
 \leq \bigg(\sum_{j=1}^n|a_{ij}|^q|y_j|^{p^*}\bigg)^{2/q} \cdot  \bigg(\sum_{j=1}^n|y_j|^{\frac{2q-2p^*}{q-2}}\bigg)^{(q-2)/q}.
\end{align*}
Since $p^*\leq 2$, we have $\frac{2q-2p^*}{q-2}\geq p^*$. Therefore, for
$y\in B_{p^*}^n$,
\[
\sum_{j=1}^n|y_j|^{\frac{2q-2p^*}{q-2}} \leq 1.
\]
This implies that for every $y\in B_{p^*}^n$ one has
$$
 \sum_{i=1}^m\bigg(\sum_{j=1}^n|a_{ij}y_j|^2\bigg)^{q/2}
 \leq  \sum_{i=1}^m \sum_{j=1}^n|a_{ij}|^q|y_j|^{p^*}
   =   \sum_{j=1}^n |y_j|^{p^*}     \sum_{i=1}^m |a_{ij}|^q
$$
$$
   =
   \sum_{j=1}^n |y_j|^{p^*}   \|(a_{ij})_{i=1}^m\|_q^q \leq
   \max_{j \leq n} \|(a_{ij})_{i=1}^m\|_q^q,
$$
which completes the proof.


\end{proof}

Now we estimate the expectation from Theorem~\ref{thm:GR}. The proof
is based on the Gaussian concentration, Theorem~\ref{thm:concentration}, and is similar to
 Theorem~2.1 and Remark~2.2 in \cite{vanH}.

\begin{lem}\label{lem: estimate expectation max pnorm X}
	Let $m,n\in \N$, $1<p^*\leq 2\leq q$, and  for $i\leq m$ let $X_i=(a_{ij}g_{ij})_{j=1}^n$. Then
\begin{align*}
	\Big(\E \max_{ i \leq m} \| X_i \|_p^q\Big)^{1/q}&
   \leq  2  \,   \max_{ i \leq m}\E \|X_i\|_p + C\, \gamma_q\, \E
    \max_{i\leq m\atop j\leq n}|a_{ij}g_{ij}| \\
   & \leq 2\, \gamma_p\, \max_{ i \leq m} \|(a_{ij})_{j=1}^n\|_p
      + C\, \gamma_q\, \E \max_{i\leq m\atop j\leq n}|a_{ij}g_{ij}|, 	
\end{align*}
%
%
where $C$ is a positive absolute constant.
\end{lem}
\begin{proof}
	For any $a,b>0$ and $q\geq 1$, we have $a^q \leq 2^{q-1}\big( |a-b|^q+b^q\big)$.
	Thus,
	\[
	\max_{1\leq i \leq m} \|X_i\|_p^q \leq 2^{q-1}\bigg[\max_{1\leq i \leq m}\Big| \|X_i\|_p - \E\|X_i\|_p \Big|^q +
\Big(\max_{1\leq i \leq m}\E \|X_i\|_p\Big)^q \bigg].
	\]
	Taking expectation and then the $q$-th root, we obtain
	\[
	\Big( \E \max_{1\leq i \leq m} \|X_i\|_p^q\Big)^{1/q} \leq 2\, \bigg(\E\max_{1\leq i \leq m}\Big| \|X_i\|_p - \E\|X_i\|_p \Big|^q\bigg)^{1/q} + 2\max_{1\leq i \leq m}\E \|X_i\|_p .
	\]
For all $i\leq m$ and $t> 0$ by (\ref{gaus-conc}) we have
	\begin{align}\label{eq:concentration gauss}
	\Pro \Big(\big| \|X_i\|_p - \E\|X_i\|_p \big| > t\Big) \leq 2\, \exp\bigg( - \frac{t^2}{2 \max_{j \leq n}|a_{ij}|^2} \bigg).
	\end{align}
	By permuting the rows of $(a_{ij})_{i,j=1}^{m,n}$, we can assume that
$$
  \max_{j \leq n}|a_{1j}| \geq \dots \geq \max_{ j \leq n}|a_{nj}|.
$$
For each $i\leq m$,  choose $j(i)\leq n$ such that $|a_{ij(i)}| = \max_{ j \leq n}|a_{ij}| $. Clearly,
	\[
	\max_{i\leq m\atop j\leq n} |a_{ij}g_{ij}| \geq \max_{ i \leq m}
|a_{ij(i)}|\cdot|g_{ij(i)}|
	\]
	and hence, by independence of $g_{ij}$'s and Lemma~\ref{max-gaus},
	\begin{align*}
  B:=	\E \max_{i\leq m\atop j\leq n} |a_{ij}g_{ij}| & \geq
       \E \max_{i \leq m} |a_{ij(i)}|\cdot |g_i|
   \geq c \max_{i \leq m} \sqrt{\log (i+3)}\cdot|a_{ij(i)}|,
	\end{align*}
	where the latter inequality follows since $|a_{1j(1)}| \geq \dots \geq |a_{nj(n)}|$.
	Thus, for $i\leq m$,
	\begin{align*}
	\max_{ j \leq n}|a_{ij}|^2 = a_{ij(i)}^2 \leq \frac{B^2}{c\log(i+3)}.
	\end{align*}
By \eqref{eq:concentration gauss} we observe for every $t>0$,
$$
	\Pro\Big(\max_{1\leq i \leq m}\big| \|X_i\|_p - \E \|X_i\|_p \big| >t \Big)
	\leq 2\, \sum_{i=1}^m \exp\bigg( -\frac{ct^2 \log(i+3)}{2 B^2} \bigg)
$$
$$
	 = 2\, \sum_{i=1}^m \bigg( \frac{1}{i+3}\bigg)^{ct^2/2 B^2}
	 \leq 2\, \int_3^\infty x^{-ct^2/2 B^2} \, dx
	 \leq  6\cdot 3^{-ct^2/2 B^2},
$$
	whenever $ct^2/B^2\geq 4$. Integrating the tail inequality proves that
	\[
	\bigg(\E \max_{1\leq i \leq m}\Big| \|X_i\|_p -
    \E \|X_i\|_p \Big|^q\bigg)^{1/q} \leq C_1 \sqrt{q}\, B  \leq C_2 \,
    \gamma_q\, \,   \E\max_{i\leq m\atop j\leq n} |a_{ij}g_{ij}|.
	\]
By the triangle inequality we obtain the first desired inequality,
the second one follows by (\ref{norm-p}).
\end{proof}

We are now ready to present the proof of the main theorem.

\begin{proof}[Proof of Theorem \ref{thm:main}]
	First observe that
	\[
	\E \,\big\|G:\ell_{p^*}^n \to \ell_q^m\big\| \leq \Big(\E\, \big\|G:\ell_{p^*}^n \to \ell_q^m\big\|^q\Big)^{1/q} = \bigg(\E \sup_{y\in B_{p^*}^n} \sum_{i=1}^m\big|\langle X_i,y\rangle\big|^q\bigg)^{1/q}.
	\]
	We have
	\begin{align*}
	 \E  \sup_{y\in B_{p^*}^n} \sum_{i=1}^m\big|\langle X_i,y\rangle\big|^q
	& \leq \E  \sup_{y\in B_{p^*}^n}\left[ \sum_{i=1}^m\big|\langle X_i,y\rangle\big|^q - \E\big|\langle X_i,y\rangle\big|^q\right] + \sup_{y\in B_{p^*}^n}\sum_{i=1}^m \E\big|\langle X_i,y\rangle\big|^q \\
	& = m\cdot \E  \sup_{y\in B_{p^*}^n}\left[ \frac{1}{m}\sum_{i=1}^m\big|\langle X_i,y\rangle\big|^q - \E\big|\langle X_i,y\rangle\big|^q\right] + m\cdot \sigma^q.
	\end{align*}
	Hence, Theorem \ref{thm:GR} applied with $E=\ell_{p^*}^n$ implies
	\begin{align*}
	\E \,\big\|G:\ell_{p^*}^n \to \ell_q^m\big\|^q & \leq m\cdot\big[ B^2 + B\sigma^{q/2}\big] +m\cdot \sigma^q
	 \leq 2 m\, \big(B^2+\sigma^{q}\big),
	\end{align*}
where $B$ and $\sigma$ are defined in that theorem.
	Therefore,
	\begin{align*}
	\Big(\E\, \big\|G:\ell_{p^*}^n \to \ell_q^m\big\|^q \Big)^{1/q}
	& \leq 2^{1/q}m^{1/q}\, \left(B^{2/q} + \sigma \right).
	\end{align*}	
Now,  recall that $T_2(\ell_p^n)\approx\sqrt{p}$ and that $B_{p^*}^n$
has modulus of convexity of power type 2 with $\lambda^{-2}\approx 1/p$.
	Therefore,
	\begin{align*}\label{eq:value of A}
	B^{2/q} &= C^{2/q} \lambda^{8/q} \, T_2^{2/q} (\ell_p^n)\left(\frac{\log m}{m}\right)^{1/q } \Big(
   \E \max_{1\leq i\leq m}\|X_i\|_{p}^q\Big)^{1/q} \\ &=
   C^{2/q}p^{5/q}(\log m)^{1/q}m^{-1/q} \Big( \E \max_{1 \leq i \leq m} \| X_i \|_p^q\Big)^{1/q}.
	\end{align*}
Applying Lemma~\ref{lem:estimate sigma}, we obtain
	\begin{align*}
	& \Big(\E\, \big\|G:\ell_{p^*}^n \to \ell_q^m\big\|^q \Big)^{1/q} \\
	& \leq (2C^2)^{1/q}\cdot p^{5/q}\cdot (\log m)^{1/q} \Big( \E \max_{1 \leq i \leq m} \| X_i \|_p^q\Big)^{1/q} + 2^{1/q}\gamma_q\cdot \max_{1 \leq j \leq n}\|(a_{ij})_{i=1}^m\|_q .
%
%
	\end{align*}
%
%
The desired bound follows now from Lemma~\ref{lem: estimate expectation max pnorm X}.
%
%
\end{proof}

{\bf Acknowledgment.} Part of this work was done while A.L. visited J.P. at the Johannes Kepler University in Linz (supported by FWFM 1628000). We would also like to thank our colleague R. Adamczak for helpful comments.

\bibliographystyle{plain}
\bibliography{random_matrices}

\end{document}